\documentclass[number,sort&compress]{elsarticle}

\usepackage{amssymb,amsmath,amsthm,url,placeins}
\usepackage[hidelinks]{hyperref}
\usepackage{cleveref}

\crefname{equation}{}{}
\Crefname{equation}{Equation}{Equations}
\crefname{theorem}{Theorem}{Theorems}
\Crefname{theorem}{Theorem}{Theorems}
\crefname{lemma}{Lemma}{Lemmas}
\Crefname{lemma}{Lemma}{Lemmas}
\crefname{proposition}{Proposition}{Propositions}
\Crefname{proposition}{Proposition}{Propositions}
\crefname{corollary}{Corollary}{Corollaries}
\Crefname{corollary}{Corollary}{Corollaries}
\crefname{conjecture}{Conjecture}{Conjectures}
\Crefname{conjecture}{Conjecture}{Conjectures}
\crefname{section}{Section}{Sections}
\Crefname{section}{Section}{Sections}
\crefname{example}{Example}{Examples}
\Crefname{example}{Example}{Examples}
\crefname{problem}{Problem}{Problems}
\Crefname{problem}{Problem}{Problems}
\crefname{table}{Table}{Tables}
\Crefname{table}{Table}{Tables}
\crefname{remark}{Remark}{Remarks}
\Crefname{remark}{Remark}{Remarks}
\crefname{definition}{Definition}{Definitions}
\Crefname{definition}{Definition}{Definitions}

\newcommand{\ZZ}{\mathbb{Z}}

\newcommand{\CC}{\mathbb{C}}

\newcommand{\ii}{\mathit{i}}

\newtheorem{theorem}{Theorem}[section]
\newtheorem{lemma}[theorem]{Lemma}

\newtheorem{corollary}[theorem]{Corollary}

\theoremstyle{definition}

\newcommand{\doi}[1]{\href{http://dx.doi.org/#1}{\texttt{doi:#1}}}
\newcommand{\arxiv}[1]{\href{http://arxiv.org/abs/#1}{\texttt{arXiv:#1}}}

\title{Tur\'{a}n numbers of $r$-graphs on $r+1$ vertices}

\author{Alexander Sidorenko}
\ead{sidorenko.ny@gmail.com}
\address{Department of Extremal Combinatorics, 
R\'{e}nyi Institute, Budapest, Hungary}

\date{\today}

\begin{document}

\begin{abstract}
Let $H_k^r$ denote an $r$-uniform hypergraph 
with $k$ edges and $r+1$ vertices, 
where $k \leq r+1$ 
(it is easy to see that such a hypergraph 
is unique up to isomorphism). 
The known general bounds on its Tur\'{a}n density are
$\pi(H_k^r) \leq \frac{k-2}{r}$ for all $k \geq 3$, 
and $\pi(H_3^r) \geq 2^{1-r}$ for $k=3$. 
We prove that 
$\pi(H_k^r) \geq 
(C_k - o(1)) \, r^{-(1+\frac{1}{k-2})}$ 
as $r\to\infty$. 
In the case $k=3$, we prove  
$\pi(H_3^r) \geq (1.7215 - o(1)) \, r^{-2}$ 
as $r\to\infty$, 
and 
$\pi(H_3^r) \geq r^{-2}$ for all $r$. 
\end{abstract}

\begin{keyword}
Tur\'{a}n-type problem \sep Tur\'{a}n density 
\sep three edge problem 
\MSC[2010]{05C35, 05C65}
\end{keyword}

\maketitle

\section{Introduction}

Let $H$ be an $r$-uniform hypergraph, or shortly $r$-graph. 
An $r$-graph $G$ is called $H$-\emph{free} 
if it does not have subgraphs isomorphic to $H$. 
Let ${\rm ex}(n,H)$ denote the \emph{Tur\'{a}n number} of $H$ 
which is the largest number of edges 
in an $H$-free $r$-graph with $n$ vertices. 
The \emph{Tur\'{a}n density} is 
$\pi(H) = \lim_{n\to\infty} {\rm ex}(n,H) 
\left/\binom{n}{r}\right.$. 

When $r>2$, the exact value of $\pi(H)$ 
(if $\pi(H)>0$)
is known only for very few $r$-graphs $H$. 
We refer the reader to 
the comprehensive survey 
by Keevash \cite{Keevash:2011}. 

Brown, Erd\H{o}s, and S\'{o}s \cite{Brown:1973} 
initiated the study of $f^{(r)}(n;v,e)$ 
which is the smallest $f$ such that every $r$-graph with $n$ vertices 
and $f$ edges has a subset of $v$ vertices with at least $e$ edges. 
The literature on this topic is extensive. 
See, for example, 
\cite{Alon:2006,Delcourt:2022,Glock:2019,Glock:2022,Ruzsa:1976}.

Let $H_k^r$ denote an $r$-graph with 
$r+1$ vertices and $k \leq r+1$ edges 
(for fixed $r$ and $k$, 
all such $r$-graphs are isomorphic). 
It is easy to see that 
${\rm ex}(n,H_k^r) = f^{(r)}(n;r+1,k) - 1$. 
The $r$-graph $H_k^r$ is also known as 
a $(k,k-1)$-\emph{daisy}. 
For results and conjectures 
on Tur\'{a}n densities of $(s,t)$-daisies in general, 
see \cite{Bollobas:2011,Ellis:2022}. 

Let $G$ be an $r$-graph with vertex set $V$ and edge set $E$. 
Let $v \in V$ be a vertex of the largest degree $d$. 
Consider the link graph $G_v$ which is an $(r-1)$-graph 
with vertex set $V\backslash\{v\}$ 
and edge set 
$\{ A \subseteq V\backslash\{v\} : \: |A| \! = \! r-1, \, (A \cup \{v\}) \in E \}$. 
Then the number of edges in 
$G_v$ is $d \geq \frac{r}{n} |E|$, where $n=|V|$. 
If $k \leq r$, and $G$ is $H_k^r$-free, 
then $G_v$ is $H_k^{r-1}$-free. 
Thus, 
${\rm ex}(n-1,H_k^{r-1}) \geq \frac{r}{n}\, {\rm ex}(n,H_k^r)$ 
which yields 
\begin{equation}\label{eq:monoton}
  \pi(H_k^{r-1}) \:\geq\: \pi(H_k^r) \, .
\end{equation}

As $H_3^2$ is simply the triangle graph, 
Mantel's theorem~\cite{Mantel:1907} establishes that 
$\pi(H_3^2) = \frac{1}{2}$. 
Frankl and F\"{u}redi~\cite{Frankl:1984} conjectured that 
$\pi(H_3^3) = \frac{2}{7}$. 
Gunderson and Semeraro~\cite{Gunderson:2017} proved 
$\pi(H_3^4) = \frac{1}{4}$. 
For exact values of ${\rm ex}(n,H_3^4)$, 
see \cite{Belkouche:2020}. 
Gunderson and Semeraro~\cite{Gunderson:2022} proved 
$\pi(H_3^6) \geq \frac{9}{64}$, 
$\pi(H_3^7) \geq \frac{35}{2^{11}}$, 
$\pi(H_3^8) \geq \frac{315}{2^{14}}$. 
By using \cref{eq:monoton}, 
a lower bound on $\pi(H_3^5)$ can be derived: 
$\pi(H_3^5) \geq \pi(H_3^6) \geq \frac{9}{64}$. 
The best general bounds 
are $\pi(H_3^r) \leq \frac{1}{r}$ 
(see \cite{Gunderson:2017}) 
and $\pi(H_3^r) \geq 2^{1-r}$ 
(see \cite{Frankl:1984}). 
We improve the lower bounds on $\pi(H_3^r)$ 
for all $r \geq 7$. 
In particular, we show 
that $\pi(H_3^r) \geq \frac{1}{r^2}$ 
(see \cref{th:expectation}). 
A stronger bound for any $r \geq 7$ can be derived 
from \cref{th:C_recurs}. 
For example, we get $\pi(H_3^7) \geq 0.0348$. 
\Cref{th:1.7215} provides 
a better asymptotic bound 
$\pi(H_3^r) \geq (1.7155 - o(1)) r^{-2}$ as $r\to\infty$. 

When $k \geq 4$, results are scarce. 
The upper bound $\pi(H_k^r) \leq \frac{k-2}{r}$ is well known 
(see, for instance, \cite[Corollary 3.9]{Markstrom:2021}). 
It is worth mentioning that $\pi(H_5^4) \geq \frac{11}{16}$ 
(see \cite{deCaen:1988,Giraud:1990}), 
$\pi(H_4^3) \geq \frac{5}{9}$ 
and 
$\pi(H_{r+1}^r) \geq 1 - (\frac{1}{2} + o(1)) \frac{\ln{r}}{r}$ 
as $r\to\infty$ 
(see \cite{Sidorenko:1995,Sidorenko:1997}). 

The best construction we know for $r=k=4$ is the following. 
Let $G_1$ be a $4$-graph with $4$ vertices and $1$ edge.
Having $G_n$, we construct $G_{n+1}$ 
by taking two disjoint copies of $G_n$ 
and adding all edges that include two vertices from each of the copies. 
When $n\to\infty$, this construction yields $\pi(H_4^4) \geq \frac{3}{7}$. 
On the other hand, 
the general upper bound $\pi(H_k^r) \leq \frac{k-2}{r}$ 
yields $\pi(H_4^4) \leq \frac{1}{2}$. 

Consider $r$ independent random points 
uniformly distributed 
on the circle of unit circumference. 
R\'{e}nyi \cite{Renyi:1953} proved that 
the expectation of the $m$th smallest distance 
(in arc length) between adjacent points 
is equal to 
$a_{r,m} = \frac{1}{r}\sum_{i=1}^m \frac{1}{r+1-i}$. 
Let random variable $\xi_{r,t}$ be the length 
of the shortest arc that 
contains at least $t$ out of the $r$ random points. 
Denote $e_{r,t} = {\mathbf E}[\xi_{r,t}]$. 
Then $e_{r,2} = a_{r,1} = \frac{1}{r^2}$ 
and $e_{r,r} = 1-a_{r,r} = 
1 - \frac{1}{r}\sum_{i=1}^r \frac{1}{i} 
\approx 1 - \frac{\ln r}{r}$. 
Theoretically, $e_{r,t}$ can be computed by 
integration (its value is equal to the volume 
of an $r$-dimensional polytope). 
In reality, 
such computations become tedious starting $r=5$. 
However, the asymptotic behavior of $e_{r,t}$ 
can be easily deduced from known results. 
Cressie \cite[Theorems~3.2, 3.3]{Cressie:1977} proved that
${\mathbf P}[r^{1+1/m} \, \xi_{r,m+1} \geq x]
\to \exp(-x^m / m!)$ as $r\to\infty$.
As the convergence holds uniformly in $x$ 
(see \cite[Section 0.1]{Reiss:1989}), 
the expectation of $r^{1+1/m} \xi_{r,m+1}$ 
converges to 
\begin{align*}
 \int_0^{\infty} \exp(-x^m / m!) \, dx 
 & \: = \: 
 \int_0^{\infty} \frac{1}{m} \, (m!)^{1/m} \, y^{\frac{1-m}{m}} \, e^{-y} \, dy 
 \\ &
 \: = \: \frac{1}{m} \, (m!)^{1/m} \, \Gamma(  {\textstyle \frac{1}{m}}) 
 \: = \:                (m!)^{1/m} \, \Gamma(1+{\textstyle \frac{1}{m}}) \, .
\end{align*}
Therefore,
\begin{align}\label{eq:asympt}
  e_{r,m+1} \: = \: 
  \left( (m!)^{1/m} \, \Gamma(1+{\textstyle \frac{1}{m}}) + o(1)\right) 
    r^{-(1+\frac{1}{m})}
  \;\;\;\;{\rm as}\;r\to\infty \, .
\end{align}
We prove that $\pi(H_k^r) \geq e_{r,k-1}$ 
(\cref{th:expectation}). 

\medskip

We call an $r$-graph \emph{pair-covering} 
if each pair of its vertices is contained 
in at least one of the edges. 
For an $r$-graph $H$, 
we define its \emph{line graph} $L(H)$ 
as a $2$-graph whose vertices are the edges of $H$. 
Edges $e'$ and $e''$ of $H$ 
are adjacent vertices in $L(H)$ 
when $|e' \cap e''| = r-1$. 
For example, 
the line graph of $H_k^r$ 
is a complete graph on $k$ vertices. 
We prove 
that for any pair-covering $r$-graph $H$, 
$\pi(H) \geq (\chi(L(H))-1) \, \pi_r$ 
where $\chi(L(H))$ is 
the chromatic number of the line graph $L(H)$, 
and 
$\pi_r = \max_n n^{-r-1} r! \binom{n}{r}
\geq \frac{2e^{-1}}{r^2+r}$
(see \cref{th:pair,th:pi}). 

This paper is organized as follows. 
In \cref{sec:pair}, 
we prove the general lower bound on $\pi(H)$ 
for pair-covering $r$-graphs $H$. 
In \cref{sec:basic}, 
we construct $H_k^r$-free $r$-graphs 
that yield $\pi(H_k^r) \geq e_{r,k-1}$. 
When $k \geq 3$, \:$H_k^r$ is pair-covering, 
so for each $n$, 
\:$r! n^{-r}{\rm ex}(n,H_k^r)$ 
provides a lower bound on $\pi(H_k^r)$. 
It allows us to obtain, for any fixed $r$, 
a bound that is better than 
$\pi(H_3^r) \geq \frac{1}{r^2}$. 
In \cref{sec:recurs}, 
we prove an asymptotic bound  
$\pi(H_3^r) \geq (1.7215 - o(1)) \, r^{-2}$ 
as $r\to\infty$.

\section{Lower bound on Tur\'{a}n density for pair-covering $r$-graphs}\label{sec:pair}

\begin{theorem}\label{th:chrom}
For any $r$-graph $H$, 
${\rm ex}(n,H) \geq \frac{\chi(L(H))-1}{n} \binom{n}{r}$. 
\end{theorem}

\begin{proof}[\bf{Proof}]
Set $t=\chi(L(H))-1$.
Let $E_j$ denote a collection of all $r$-element subsets 
$\{i_1,\ldots,i_r\}\subseteq\ZZ_n$ such that 
$i_1+\ldots+i_r \equiv j \;{\rm mod}\; n$. 
As $|E_0|+\ldots+|E_{n-1}| = \binom{n}{r}$, 
there exist pairwise distinct $j_1,\ldots,j_t$ such that 
$|E_{j_1}|+\ldots+|E_{j_t}| \geq \frac{t}{n} \binom{n}{r}$. 
Let $G$ be an $r$-graph with vertex set $\ZZ_n$ 
and edge set $E_{j_1} \cup\cdots\cup E_{j_t}$. 
We claim that $G$ is $H$-free. 
Indeed, suppose $G$ contains a copy of $H$. 
Color each edge $e$ of that copy 
with color $\sum_{i \in e} i$. 
Notice that except for $j_1,\ldots,j_t$, 
no other colors have been used. 
If edges $e',e''$ are adjacent vertices 
in the line graph $L(H)$, 
that is $|e' \cap e''| = r-1$, 
then $e'$ and $e''$ are assigned different colors. 
This contradicts the assumption that 
the chromatic number of $L(H)$ is $t+1$. 
\end{proof}

\begin{lemma}\label{th:blow}
For any pair-covering $r$-graph $H$, 
$\:\pi(H) \!\geq\! r! \, n^{-r} {\rm ex}(n,H)\:$ 
holds for all $n$. 
\end{lemma}

\begin{proof}[\bf{Proof}]
Let $G_n$ be an $H$-free $r$-graph 
with $n$ vertices and ${\rm ex}(n,H)$ edges.
Construct graph $G_{mn}$ as a blow-up of $G_n$ by a factor of $m$. 
Then $G_{mn}$ has $mn$ vertices and 
$m^r {\rm ex}(n,H)$ edges. 
Since $G_n$ is $H$-free, 
and $H$ is pair-covering, 
$G_{mn}$ is also $H$-free. 
Hence, 
$
  \pi(H) \geq\
  \lim_{m\to\infty} m^r {\rm ex}(n,H) 
  \left/ \binom{mn}{r} \right. 
  = r! \, n^{-r} {\rm ex}(n,H)
  .
$
\end{proof}

The next statement is a consequence of 
\cref{th:chrom,th:blow}. 

\begin{corollary}\label{th:pair}
For any pair-covering $r$-graph $H$, 
$\pi(H) \geq (\chi(L(H))-1) \, \pi_r$ 
where 
$\pi_r = \max_n n^{-r-1} r! \binom{n}{r}$. 
\end{corollary}

\begin{theorem}\label{th:pi}
$\frac{2e^{-1}}{r^2+r} \leq \pi_r \leq \frac{2e^{-1}}{r^2-r}\,$.
\end{theorem}

\begin{proof}[\bf{Proof}]
Set $n = (r^2+r)/2$.
If $j<r$, then $(j+1)j \leq 2n$, hence, 
\[
  1 - \frac{j}{n}
  \:=\: 1 - \frac{j+1}{n} + \frac{1}{n}
  \:\geq\: 1 - \frac{j+1}{n} + \frac{(j+1)j}{2n^2}
  \:\geq\: \left(1 - \frac{1}{n}\right)^{j+1} 
  ,
\]
and 
\begin{align*}
  \pi_r & \:\geq\: \frac{r!}{n^{r+1}} \binom{n}{r}
  \:=\: \frac{1}{n} \, \prod_{j=1}^{r-1} \left(1 - \frac{j}{n}\right)
  \:\geq\: 
    \frac{1}{n} \, \prod_{j=1}^{r-1} \left(1 - \frac{1}{n}\right)^{j+1}
  \\ & 
  \:=\: \frac{1}{n} \left(1 - \frac{1}{n}\right)^{-(r^2+r-2)/2}
  =\: \frac{1}{n} \left(1 - \frac{1}{n}\right)^{n-1}
  >\: \frac{e^{-1}}{n} 
  \:=\: \frac{2e^{-1}}{r^2+r}
  \, .
\end{align*}
To prove the upper bound, set $x(n) = \frac{r^2-r}{2n}$. 
Then 
\begin{align*}
  \frac{r!}{n^{r+1}} \binom{n}{r} 
  & \: = \: \frac{1}{n} \prod_{j=1}^{r-1} \left(1 - \frac{j}{n}\right)
  \:\leq\: \frac{1}{n} \prod_{j=1}^{r-1} \left(1 - \frac{1}{n}\right)^j
  \: = \: \frac{1}{n} \left(1 - \frac{1}{n}\right)^{\binom{r}{2}}
  \\ & \: = \:\frac{1}{n} \left(
    \left(1 - \frac{1}{n}\right)^n
  \right)^{\binom{r}{2}/n}
  < \: \frac{1}{n} \, e^{-\binom{r}{2}/n}
  \: = \: \frac{2}{r^2-r} \, x(n) \, e^{-x(n)} \, .
\end{align*}
Since the maximum of $x e^{-x}$ is reached at $x=1$, we get 
$\pi_r \leq \frac{2e^{-1}}{r^2-r}$. 
\end{proof}

\section{Construction of $H_k^r$-free $r$-graphs}\label{sec:basic}

When $k \geq 3$, \:$H_k^r$ is pair-covering 
and $L(H_k^r)=K_k$, 
so \cref{th:pair} could be used. 
However, we can add more edges 
to the construction of \cref{th:chrom} 
without creating copies of $H_k^r$. 

We are going to define infinite $r$-graph $G_{r,t}$ 
where any $r+1$ vertices induce at most $t$ edges. 
The vertex set of $G_{r,t}$ is 
the unit circle in the complex plane: 
${\mathbf C}=\{z\in\CC: |z|=1\}$.  
For a finite subset $A \subset {\mathbf C}$, 
let $\Pi(A)$ denote the product of its elements, 
and let $\Delta_t(A)$ denote the length 
of the shortest arc in ${\mathbf C}$ 
that contains at least $t$ elements of $A$. 
An $r$-element subset $A\subset{\mathbf C}$ 
is an edge of $G_{r,t}$ if 
$\Pi(A) = e^{\ii\varphi}$ 
with $\varphi \in [0,\Delta_t(A)]$. 

\begin{theorem}\label{th:T0}
$G_{r,k-1}$ is $H_k^r$-free. 
\end{theorem}

\begin{proof}[\bf{Proof}]
Suppose to the contrary that 
there exists a subset 
$B\subset{\mathbf C}$ of size $r+1$ 
with $k$ distinct elements 
$z_1,z_2,\ldots,z_k \in B$ 
such that 
$B \backslash \{z_j\}$ is an edge of $G_{r,k-1}$ 
for each $j=1,2,\ldots,k$. 
Select $\varphi_j \in [0,2\pi)$ such that 
$e^{\ii\varphi_j} = \Pi(B) / z_j$. 
We may assume that 
$\varphi_1 < \ldots < \varphi_{k-1} < \varphi_k$. 
Then 
$z_j = \Pi(B) e^{-i\varphi_j} = 
z_{k-1} e^{\ii(\varphi_{k-1} - \varphi_j)}$. 
As 
$0 \leq \varphi_j \leq \varphi_{k-1}$ for $j=1,\ldots,k-1$, 
the arc 
$\{z_{k-1} e^{\ii\varphi}: \varphi \in [0,\varphi_{k-1}]\}$ 
contains $z_1,\ldots,z_{k-1}$. 
Hence, $\Delta_{k-1}(B\backslash\{z_k\}) \leq \varphi_{k-1}$. 
On the other hand, since $B\backslash\{z_k\}$ is an edge, 
$\Delta_{k-1}(B\backslash\{z_k\}) \geq \varphi_k > \varphi_{k-1}$, 
which is a contradiction. 
\end{proof}

We defined $G_{r,t}$ in such a way that 
$r$ independent random points 
uniformly distributed on ${\mathbf C}$ 
form an edge with probability $e_{r,t}\,$. 
Recall that $e_{r,2} = \frac{1}{r^2}$. 
The asymptotic expression for $e_{r,t}$ 
is given by \cref{eq:asympt}. 

\begin{corollary}\label{th:expectation}
$\pi(H_k^r) \geq e_{r,k-1}$. 
In particular, 
$\pi(H_3^r) \geq \frac{1}{r^2}$\,. 
\end{corollary}

Now we are going to construct finite $H_k^r$-free $r$-graphs. 
They are essentially subgraphs of $G_{r,t}$. 

Let $n \geq 3$. 
Consider an oriented $n$-vertex cycle $C_n$. 
We associate its vertices with the elements of $\ZZ_n$ 
so that $(x,x+1)$ for each $x\in\ZZ_n$ is an arc. 
For a subset $A=\{x_1,\ldots,x_s\}\subseteq\ZZ_n$, 
we call a sequence $(x_0,x_1,\ldots,x_s=x_0)$ 
a \emph{circular representation} of $A$ 
if for each $i=0,1,\ldots,s-1$,
the directed path $(x_i,x_{i+1})$ does not contain other points $x_j$ 
except $x_i$ and $x_{i+1}$.
Let $d(A)$ denote the \emph{diameter} of $A$, 
that is the length of the directed path in $C_n$ 
which passes through all elements of $A$. 
For $t \leq |A|$, we set 
$d_t(A) := \min \{d(B)|\: B \subseteq A,\: |B|=t\}$. 
In particular, $d_2(A)$ is the minimum distance between 
two consecutive elements in a circular representation of $A$. 

For $j\in\ZZ_n$, let $G(n,r,t,j)$ denote 
an $r$-graph with vertex set $\ZZ_n$ 
whose edges are $r$-element subsets $A$ of $\ZZ_n$ 
such that 
$(j + \sum_{x \in A} x) \in\{0,1,\ldots,d_t(A)\}$. 
The proof of the following statement is very similar 
to the proof of of \cref{th:T0}. 

\begin{theorem}\label{th:T1}
$G(n,r,k-1,j)$ is $H_r^k$-free. 
\end{theorem}

Let $N(n,r,t,d)$ denote the number of 
$r$-element subsets $A\subseteq\ZZ_n$ 
such that $d_t(A) \geq d$.

\begin{corollary}\label{th:C1}
\begin{align}\label{eq:C1}
  {\rm ex}(n,H_k^r) \:\geq\:  
    \frac{1}{n} 
      \sum_{d=0}^{\left\lfloor\!\frac{(k-2)n}{r}\!\right\rfloor} 
        N(n,r,k-1,d) \, ,
\end{align}
\end{corollary}

\begin{proof}[\bf{Proof}]
Each $r$-element subset $A\subseteq\ZZ_n$ 
appears as an edge in $d_{k-1}(A)+1$ $\,r$-graphs 
among $G(n,r,k-1,0),G(n,r,k-1,1),\ldots,G(n,r,k-1,n-1)$. 
Thus, the average number of edges in these $r$-graphs is equal to 
\begin{align*}
     & \frac{1}{n} \sum_{d=0}^n 
       (d+1)\big(N(n,r,k-1,d) - N(n,r,k-1,d+1)\big) \\
= \: & \frac{1}{n} \sum_{d=0}^n 
       N(n,r,k-1,d) \, . 
\end{align*}
As $d_{k-1}(A) \leq \frac{(k-2)n}{r}$,
we get 
$N(n,r,k-1,d)=0$ for $d > \frac{(k-2)n}{r}$. 
\end{proof}

\begin{lemma}\label{th:L2}
$N(n,r,2,d) = \frac{n}{r} \binom{n-1-(d-1)r}{r-1}\:$ 
for $1 \leq d \leq \frac{n-r}{r}$. 
\end{lemma}

\begin{proof}[\bf{Proof}]
Consider first a similar problem on an interval instead of a circle. 
Let $P(n,r,d)$ denote the number of integer sequences 
$1 \leq i_1 < i_2 < \ldots < i_r \leq n$ 
such that $i_j-i_{j-1} \geq d$ for $j=2,\ldots,r$. 
It is well known 
(and easy to prove by induction on $r$) 
that $P(n,r,d) = \binom{n-(r-1)(d-1)}{r}$.

Now consider the circular problem. 
When we fix one of the points 
of an $r$-element subset $A\subset\ZZ_n$ with $d_2(A) \geq d$,
the remaining $r-1$ points of $A$ are contained in the interval 
that consists of $n-1-2(d-1)$ points. 
Thus, for a fixed element $j\in\ZZ_n$, 
the number of $r$-element subsets $A\subseteq\ZZ_n$\,, 
such that $j \in A$ and $d_2(A) \geq d$, 
is equal to $P\big(n-(2d-1),r-1,d\big)$. 
When we multiply this number 
by the number of choices $j\in\ZZ_n$, 
every $r$-element subset $A$ with $d_2(A) \geq d$ 
is counted $r$ times. 
Therefore, 
$N(n,r,2,d) = \frac{n}{r}P\big(n-(2d-1),r-1,d\big)
= \frac{n}{r} \binom{n-1-(d-1)r}{r-1}$. 
\end{proof}

As $N(n,r,2,0) = \binom{n}{r}$ and 
\begin{align*}
    \sum_{d=1}^{\lfloor n/r \rfloor} 
      \binom{n-1-r(d-1)}{r-1} 
    \: = \: 
    \sum_{d=0}^{\lfloor n/r \rfloor - 1} 
      \binom{n-1-rd}{r-1} \, ,
\end{align*}
we get

\begin{corollary}\label{th:k=3}
\begin{align}\label{eq:C3a}
  {\rm ex}(n,H_3^r) & \:\geq\:  
    \frac{1}{n} \binom{n}{r} \: + \:
    \frac{1}{r} 
    \sum_{d=0}^{\lfloor n/r \rfloor - 1} 
      \binom{n-1-rd}{r-1} \, .
\end{align}
\end{corollary}

Another way to prove the bound 
${\rm ex}(n,H_3^r) \geq \frac{1}{r^2}$ 
is to derive it directly from \cref{eq:C3a}. 
Indeed, as 
$\binom{m}{r-1} \geq \frac{1}{r} \sum_{j=0}^{r-1} \binom{m-j}{r-1}$, 
we get 
\begin{align*}
  {\rm ex}(n,H_3^r) & \:\geq\:  
    \frac{1}{r} 
    \sum_{d=0}^{\lfloor n/r \rfloor - 1} 
      \binom{n-1-rd}{r-1}
  \:\geq\: 
    \frac{1}{r^2} 
    \sum_{i=0}^{n-r+1} \binom{n-i}{r-1}
    \: = \: \frac{1}{r^2} \binom{n}{r}
  \, .
\end{align*}

\medskip

Since $H_3^r$ is pair-covering,  
we can improve on $\pi(H_3^r) \geq \frac{1}{r^2}$ 
by using \cref{th:k=3} together with \cref{th:blow}. 
This method works best when $n \approx 0.656 \, r^2$. 
For example, when $k=3$, $r=7$, 
the choice of $n=33$ in \cref{eq:C3a} yields 
${\rm ex}(33,H_3^7) \geq 288334$, 
and by \cref{th:blow}, $\pi(H_3^7) \geq 0.034098$. 
For comparison, $1/7^2$ is merely $0.020408...$ . 
When $k=3$, $r=8$, 
the choice of $n=42$ in \cref{eq:C3a} yields 
${\rm ex}(42,H_3^8) \geq 6217014$, 
and by \cref{th:blow}, $\pi(H_3^8) \geq 0.025888$.
This is better than the bound $\frac{315}{2^14}$ from \cite{Gunderson:2022}. 
The coefficient $0.656$ comes from the asymptotic 
approximation of the right hand side of \cref{eq:C3a}, which is 
$\frac{n^r}{r!} \cdot \frac{1}{r^2} \cdot f\left(\frac{r^2}{2n}\right)$, 
where $f(x) = 2xe^{-x} \left(1 + \frac{1}{1-e^{-2x}}\right)$. 
The function $f(x)$ reaches its maximum at $x = x_0 \approx 0.762$, 
where $f(x_0) \approx 1.6207$ and $\frac{1}{2x_0} \approx 0.656\,$. 
We skip the proof, since it is similar to the proof of \cref{th:inf} ,
and we are going to get even stronger bounds on ${\rm ex}(n,H_3^r)$ 
in the next section.

\section{Augmented construction}\label{sec:recurs}

The approach used in the preceding section 
(combining \cref{th:k=3} and \cref{th:blow}) 
leaves room for improvement: 
when we blow up the $r$-graph $G(n,r,
\linebreak
k-1,j)$, 
new edges can be added 
without creating copies of $H_k^r$. 

\begin{theorem}\label{th:T3}
If $t=k-1$, then
\begin{align}\label{eq:T3}
  {\rm ex}(tn,H_k^r) \:\geq\: t^r {\rm ex}(n,H_k^r) 
  + \, \frac{1}{n} \binom{tn}{r} 
  - \, \frac{1}{n} \, t^r \binom{n}{r} 
  - \frac{1}{2} \, t^{r-1}(t-1) \binom{n-1}{r-2} .
\end{align}
\end{theorem}

\begin{proof}[\bf{Proof}]
Denote $[t] := \{1,\ldots,t\}$. 
Let $G$ be an $H_k^r$-free $r$-graph 
with ${\rm ex}(n,H_k^r)$ edges 
whose vertex set is $\ZZ_n$. 
Consider a projection 
$p: \ZZ_n\times[t] \to \ZZ_n$ 
where $p((x,y))=x$ for $x\in\ZZ_n$, $y\in[t]$. 
For $j\in\ZZ_n$, denote by $G_j$ an $r$-graph 
with vertex set $\ZZ_n\times[t]$
and edges $A\subseteq\ZZ_n\times[t]$,
($|A|=r$) such that either
({\it i}) $|p(A)|=r$ and $p(A)$ is an edge in $G$, 
or 
({\it ii}) $|p(A)| \leq r-2$ and 
$\sum_{x \in A} p(x) = j$.

We claim that $G_j$ is $H_k^r$-free for each $j\in\ZZ_n$. 
Indeed, let $B\subseteq\ZZ_n\times[t]$, $|B|=r+1$. 
If $|p(B)|=r+1$, then 
$p(B)$ contains at most $t$ edges of $G$,
so $B$ contains at most $t$ edges of $G_j$. 
If $|p(B)|=r$, then 
$B$ contains at most $2$ edges of $G_j$ 
(namely $B\backslash\{v'\}$ and $B\backslash\{v''\}$ 
where $p(v')=p(v'')$). 
If $|p(B)| \leq r-1$, then 
$B \backslash \{v\}$ is an edge of $G_j$ 
if and only if $v \in B$ 
and $p(v) = (\sum_{x \in B} p(x)) - j$. 
There exist at most $t$ such vertices $v$. 

There are 
$\binom{tn}{r} - t^r \binom{n}{r} 
- n \binom{n-1}{r-2} \binom{t}{2} t^{r-2}$ 
$r$-element subsets of $\ZZ_n\times[t]$ 
with projection of size $r-2$ or less. 
Each of them appears as an edge 
in exactly one of the $r$-graphs $G_0,G_1,\ldots,G_{n-1}$. 
Thus, the average number of edges in these $r$-graphs 
is equal to the right hand side of \cref{eq:T3}. 
\end{proof}

By applying \cref{eq:T3} repeatedly 
and replacing $\binom{t^{i-1} n - 1}{r-2}$ with 
$\frac{r-1}{n} t^{1-i} \binom{t^{i-1} n}{r-1}$, 
we get 

\begin{corollary}\label{th:C_recurs}
If $t=k-1$, then
\begin{multline*}
  {\rm ex}(t^s n,H_k^r) \:\geq\: t^{sr} {\rm ex}(n,H_k^r) 
  \: + \: \frac{1}{n} \sum_{i=1}^{s-1} (t-1) t^{sr-ir-i} \binom{t^i n}{r} 
  \: + \: \frac{1}{n} \, t^{-s+1} \binom{t^s n}{r} 
  \\
  - \frac{1}{n} \, t^{sr} \binom{n}{r} 
  \: - \frac{r-1}{n} \sum_{i=1}^s \frac{t-1}{2} t^{sr-ir+r-i} \binom{t^{i-1} n}{r-1} .
\end{multline*}
\end{corollary}

We may use \cref{th:T3,th:C_recurs} 
to improve the bounds from \cref{sec:basic}. 
For example, when $k=3$, $r=7$, 
the choice of $n=30$ in \cref{eq:C3a} yields 
${\rm ex}(30,H_3^7) \geq 147553$.
Then by \cref{eq:T3}, we get 
${\rm ex}(60,H_3^7) \geq 19274108$,
and \cref{th:blow} gives us $\pi(H_3^7) \geq 0.034701$. 
When we use \cref{th:C_recurs} with $s=3$, $n=30$ 
and apply \cref{th:blow}, we get 
$\pi(H_3^7) \geq 0.034818$. 

\medskip

By combining \cref{th:C_recurs} with \cref{th:blow}, we get

\begin{corollary}\label{th:C_inf}
If $k\geq 3$ and $t=k-1$, then
\begin{multline}\label{eq:C_inf}
  \pi(H_k^r) \:\geq\: \frac{r!}{n^r} {\rm ex}(n,H_k^r) 
  \: - \: \frac{r!}{n^{r+1}} \, \binom{n}{r} 
  \: + \: (t-1) \, \frac{1}{n} \sum_{i=1}^{\infty}
    \frac{t^{-i} r!}{(t^i n)^r} \binom{t^i n}{r} 
  \\
  \: - \binom{t-1}{2} \, \frac{(r-1)r}{n^2} \sum_{i=1}^{\infty}
    \frac{t^{-2i} (r-1)!}{(t^{i-1} n)^{r-1}} \binom{t^{i-1} n}{r-1} .
\end{multline}
\end{corollary}

When $r\to\infty$ and $\frac{r^2}{2n} \to x = const$, 
the binomial coefficients in \cref{eq:C_inf} 
can be replaced with exponents. 
Define
\begin{align}\label{eq:F}
  F(x) = \frac{2xe^{-x}}{1-e^{-2x}} 
    + 2x \sum_{i=1}^{\infty} 2^{-i} \exp(-2^{-i}x) 
  - 4x^2 \sum_{i=1}^{\infty} 2^{-2i} \exp(-2^{1-i}x)
     . 
\end{align}
It is easy to see that the sums in \cref{eq:F} converge. 

\begin{theorem}\label{th:inf}
For any real $x>0$ and $\varepsilon > 0$, 
\begin{align}\label{eq:inf}
  \pi(H_3^r) \:\geq\: 
  \frac{F(x) - \varepsilon - o(1)}{r^2} 
  \;\;\;\;{\rm as}\; r\to\infty
  \, .      
\end{align}
\end{theorem}

\begin{lemma}\label{th:binom}
Let $c,d$, and $x>0$ be constants. 
If $n=n(r) = \lfloor r^2/(2x) \rfloor$, then 
\begin{align}\label{eq:binom}
  \binom{n-c-rd}{r} \: = \: 
  \frac{n^r}{r!} \, e^{-(2d+1)x}\,(1+o(1)) 
  \;\;\;\;{\rm as}\; r\to\infty
  \, .
\end{align}
\end{lemma}

\begin{proof}[\bf{Proof}]
When $r^2 / m \to 2x$, Stirling formula gives 
$\binom{m}{r} = \frac{m^r}{r!} \, e^{-x}\,(1+o(1))$
(see, for example, \cite[(5.8)]{Spencer:2014}). 
As $n(r)-c = \frac{r^2}{x} - O(1)$, we get 
$\binom{n-c-dr}{r} = \frac{(n-c-dr)^r}{r!} \, e^{-x}\,(1+o(1))$. 
Then \cref{eq:binom} follows from 
\begin{align*}
  (n-c-rd)^r n^{-r} 
  \: = \: \left(1-\frac{2xd}{r}\right)^r + o(1) 
  \: = \: e^{-2xd} + o(1) 
  \;\;\;\;{\rm as}\; r\to\infty
  \, .
\end{align*}
\end{proof}

\begin{proof}[\bf{Proof of \cref{th:inf}}]
Notice that 
$e^{-x}/(1-e^{-2x}) = \sum_{i=0}^{\infty} e^{-(2i+1)x}$. 
Select $m_0$ such that for all $m \geq m_0\,$, 
\begin{multline}\label{eq:19}
  \sum_{i=0}^m 2x e^{-(2i+1)x}
    \: + \: 2x \sum_{i=1}^m 2^{-i} \exp(-2^{-i}x) 
        - 4x^2 \sum_{i=1}^m 2^{-2i} \exp(-2^{1-i}x)
  \\ \geq \: F(x) - \varepsilon \, .
\end{multline}
Select $n=n(r) = \lfloor r^2/(2x) \rfloor$. 
By plugging \cref{eq:C3a} into \cref{eq:C_inf}, 
we get 
\begin{multline}\label{eq:20}
  \pi(H_3^r) \:\geq\: \frac{(r-1)!}{n^r} 
    \sum_{i=0}^{\lfloor n/r \rfloor - 1} 
      \binom{n-1-ir}{r-1} 
  \: + \: \frac{1}{n} \sum_{i=1}^{\infty}
    \frac{2^{-i} r!}{(2^i n)^r} \binom{2^i n}{r} 
  \\
  \: - \frac{(r-1)r}{n^2} \sum_{i=1}^{\infty}
    \frac{2^{-2i} (r-1)!}{(2^{i-1} n)^{r-1}} \binom{2^{i-1} n}{r-1} .
\end{multline}
Select $m \geq m_0$ large enough 
so that for each $i>m$, 
the $i$th term in the second sum in \cref{eq:20} 
is greater than the $i$th term in the third sum. 
We may assume that $r$ and $n$ are large enough 
to ensure $\lfloor n/r \rfloor - 1 \geq m$. 
Hence, we may drop all terms with $i>m$
from the three sums in \cref{eq:20}:
\begin{multline}\label{eq:21}
  \pi(H_3^r) \:\geq\: \frac{(r-1)!}{n^r} 
    \sum_{i=0}^m 
      \binom{n-1-ir}{r-1} 
  \: + \: \frac{1}{n} \sum_{i=1}^m
    \frac{2^{-i} r!}{(2^i n)^r} \binom{2^i n}{r} 
  \\
  \: - \frac{(r-1)r}{n^2} \sum_{i=1}^m
    \frac{2^{-2i} (r-1)!}{(2^{i-1} n)^{r-1}} \binom{2^{i-1} n}{r-1} .
\end{multline}
The sums on the right hand side of \cref{eq:21} are finite, 
and by using $r^2/n = 2x + o(1)$ 
and applying \cref{th:binom},
we get 
\begin{multline}\label{eq:22}
  r^2 \, \pi(H_3^r) \:\geq\:
  \sum_{i=0}^m 2x e^{-(2i+1)x} 
  \: + \: 2x \sum_{i=1}^m
            2^{-i} \exp(-2^{-i}x) 
  \\ - \: 4x^2 \sum_{i=1}^m
            2^{-2i} \exp(-2^{1-i}x)
  + o(1) \, .
\end{multline}
Bound \cref{eq:inf} follows from \cref{eq:19,eq:22}. 
\end{proof}

As $F(1.065) > 1.7215$, we get 

\begin{corollary}\label{th:1.7215}
$\pi(H_3^r) \geq (1.7215 - o(1)) \, r^{-2}$ as $r\to\infty$. 
\end{corollary}

\section*{Acknowledgements}
The author is thankful to the anonymous reviewers 
whose suggestions helped to improve the presentation 
of this paper.

\end{document}